\documentclass[12pt]{amsart}

\usepackage{amsfonts,amsmath,amsthm,amssymb,amsxtra}
\usepackage{geometry}
\usepackage{color}
\definecolor{greenbean}{RGB}{199,237,204}
\usepackage{tikz}
\usetikzlibrary{calc}

\usepackage{multicol}

\usepackage[colorlinks=true, bookmarks=true]{hyperref}

\usepackage[english]{babel}
\usepackage{indentfirst,bm}
\usepackage{latexsym}

\usepackage{graphicx}
\usepackage[tight]{subfigure}
\theoremstyle{plain}
\newtheorem{thm}{Theorem}[section]
\newtheorem{lemma}[thm]{Lemma}
\newtheorem{prop}[thm]{Proposition}
\newtheorem{cor}[thm]{Corollary}
\newtheorem{eg}[thm]{Example}

\theoremstyle{definition}
\newtheorem{Def}[thm]{Definition}

\theoremstyle{remark}
\newtheorem{remark}[thm]{Remark}

\newcommand{\A}{\mathcal{A}}

\newcommand{\MA}{\mathcal{M}_{\mathcal{A}}}

\newcommand{\Bigp}[1]{ \Big( #1\Big )}
\newcommand{\notion}[1]{{\it #1}}

\title[Classification of 9 Projective Lines]{Classification of Moduli Spaces of Arrangements of 9 Projective Lines}
\author{Fei Ye}
\address{Emmy Noether Research Institute,
 Department of Mathematics, Bar-Ilan University, Ramat Gan, 52900, Israel.}
\email{fye@macs.biu.ac.il}
\date{}

\thanks{This work was partially supported by the Oswald Veblen Fund and by the Minerva Foundation of Germany.} 



\begin{document}
\begin{abstract}In the study of line arrangements, searching for minimal examples of line arrangements whose fundamental groups are not  combinatorially invariant is a very interesting and hard problem. It was known that such a minimal arrangement must have at least 9 lines. In this paper, we extend the number to 10 by a new method. We classify arrangements of 9 projective lines according to the irreducibility of their moduli spaces and show that fundamental groups of complements of arrangements of 9 projective lines are combinatorially invariant.  The idea and results have been used to classify arrangements of 10 projective lines.
\end{abstract}

\maketitle

\section{Introduction}
A {\it hyperplane arrangement} $\mathcal{A}=\{L_1, L_2, \dots, L_n\}$  in  $\mathbb{CP}^r$ is a finite collection of hyperplanes.  We call $M(\mathcal{A})=\mathbb{CP}^r\setminus (\bigcup\limits_{L\in \mathcal{A}}L)$ the complement of $\mathcal{A}$. The set $L(\mathcal{A})=\{\bigcap\limits_{i\in S}L_i | S\subseteq\{1, 2, \dots, n\}\}$ partially ordered by reverse inclusion is called the \notion{intersection lattice} of $\mathcal{A}$. Let $\mathcal{A}_1$ and $\mathcal{A}_2$ be  two arrangements of $n$ hyperplanes. We say that intersection lattices $L(\mathcal{A}_1)$ and $L(\mathcal{A}_2)$ are isomorphic, denoted by $L(\mathcal{A}_1)\sim L(\mathcal{A}_2)$,  if there is a permutation $\phi$ of the set $\{1, 2, \dots, n \}$ such that 
\[\dim \Big( \bigcap\limits_{\substack{i\in S \\ G_i\in \mathcal{A}_1}}G_i \Big) =\dim \Bigp{ \bigcap\limits_{\substack{j\in \phi(S)\\ H_j\in \mathcal{A}_2}}H_j }\]
 for any nonempty subset $S\subseteq\{1,2,\dots, n\}$. Two arrangements are \notion{lattice isomorphic} if their lattices are isomorphic. In this paper, we only consider line arrangements in $\mathbb{CP}^2$.

An essential topic in hyperplane arrangements theory is to study the interaction between topology of complements and combinatorics of intersection lattices.  Naturally enough, one may ask how close topology and combinatorics of a given arrangement are related.
Two arrangements $\mathcal{A}_1$ and $\mathcal{A}_2$ are \notion{homeomorphic equivalent} if there is a homeomorphism between  their complements.   A more concrete question is how close lattice isomorphism and homeomorphic equivalence are being a one-to-one correspondence. 

The deepest theorem in the theory of arrangement of lines in projective space of dimension 2 is the Jiang-Yau Theorem \cite{Jiang1998} which asserts that the intersection lattice of the line arrangement is a topological invariant. It is natural to ask to what extent the converse of Jiang-Yau Theorem is true. Jiang-Yau \cite{Jiang1994} and subsequently Wang-Yau \cite{Wang2005} have shown that the converse statement is also true for a large class of line arrangements. Therefore Jiang-Yau Theorem \cite{Jiang1998} initiates a new research direction: Can one find a Zariski pair of line arrangements, i.e. a pair of arrangements which are lattice isomorphic but not homeomorphic equivalent. 

A pair of arrangements which are lattice isomorphism but not homeomorphic equivalent is called a \notion{Zariski pair}. Our definition is stronger than the definition introduced by Artal in \cite{Artal-Bartolo1994}, which we shall call \notion{weak Zariski pairs} (see \cite{Artal-Bartolo2008} for a survey on Zariski pairs). %
The first Zariski pair of arrangements was constructed by Rybnikov \cite{Rybnikov2011}. Each arrangement in Rybinikov's example consists of 13 lines and 15 triple points.   Latter, in \cite{ArtalBartolo2005}, the authors provide another (weak) Zariski pair of two arrangements $\mathcal{H}^+:=\mathcal{C}^+\cup\{N^+\}$ and $\mathcal{H}^-:\cup\{N^-\}$, where $\mathcal{C}^{\pm}$ are arrangements  (Figure \ref{fig:extended-FS}) extending Falk-Sturmfels arrangements (Figure \ref{fig:Falk-Sturmfels}), and $N^{\pm}$ are lines passing through a triple point and a double point of $\mathcal{C}^{\pm}$. The proof is based on the observation that there is no order-preserving homeomorphism between $(\mathbb{P}^2, \mathcal{C}^+)$ and $(\mathbb{P}^2, \mathcal{C}^-)$.  On the contrary direction,  Garber, Teicher and Vishne \cite{Garber2003} proved that there is no Zariski pair of arrangement of upto 8 real lines which covered the result of Fan \cite{Fan1997} on arrangements of 6 lines. This result was recently generalized to arrangements of 8 complex lines by Nazir and Yoshinaga
 \cite{Nazir2010}.  
 
 A natural question is what is the minimal number of lines of a Zariski pair of line arrangements. %This is one of our motivations.

On the other hand, it was Jiang and Yau \cite{Jiang1994} who first observe that the statement: two lattice isotopy line arrangements (i.e. they are connected by a one-parameter family with constant intersection lattice) have diffeomorphic complements, follows trivially from Teissier's numerical characterization of Whitney condition. In \cite{Jiang1994} and \cite{Wang2005}, the authors found large classes of line arrangements, called \notion{nice arrangements} and \notion{simple arrangements} respectively, whose intersection lattices determine topology of the complements.  Nazir-Yoshinaga \cite{Nazir2010} found new classes of line arrangements whose intersection lattices determine  the topology of the complements.  Unlike nice and simple arrangements whose intersection lattices have special properties, Nazir and Yashinaga's new classes require that all intersection points with multiplicity at least $3$ are in special positions. It makes Nazir and Yashinaga's results more practical to study arrangements of fewer lines. Indeed, in \cite{Nazir2010}, Nazir and Yoshinaga classify arrangements of 8 lines and present a list of classification of arrangements of 9 lines.   
 
 In this paper, we introduce new ideas to classify arrangements of lines. We prove that Nazir and Yoshinaga's list on the classification of arrangements of 9 lines is complete. As a corollary, we conclude that there is no Zariski pair of arrangements of $9$ lines.   The idea and results of this paper have been used to classify moduli spaces of arrangements of 10 projective lines (see \cite{Amram2012}).
 
The paper is organized as following. In section 1, we recall some results in Nazir and Yoshinaga's paper. In section 2, we prove  the list of the classification of arrangements of 9 lines by Nazir and Yoshinaga is complete. In section 4, we consider the example of arrangements of 10 lines $\mathcal{C}^\pm$ and give an explicit diffeomorphism between the complements $M(\mathcal{C}^\pm)$.

\paragraph{Acknowledgements:}The author is grateful to  M. Amram,  M. Cohen, M. Eliyahu, D. Garber, M. Teicher, E. Artal Bartolo,
 and J.I. Cogolludo Agust\'in for helpful conversations and comments, and especially to D. Garber for comments on a draft of this manuscript.   

\section{Simple $C_{\leq 3}$ line arrangements}

Consider the dual space $(\mathbb{CP}^2)^*$ of the projective space $\mathbb{CP}^2$. A line arrangement $\mathcal{A}=\{L_1, L_2,\dots, L_n\}$ can be viewed as an $n$-tuple of points $(L^*_1, L^*_2,\dots, L^*_n)$ in the product of the dual spaces $((\mathbb{CP}^2)^*)^n$. We define
the moduli space of arrangements with the fixed lattice $L(\mathcal{A})$ as 
\[\mathcal{M}_\mathcal{A}=\{\mathcal{B}\in ((\mathbb{CP}^2)^*)^n | L(\mathcal{B})= L(\mathcal{A})\}\Big/\text{PGL}_3(\mathbb{C})\subseteq ((\mathbb{CP}^2)^*)^n\Big/\text{PGL}_3(\mathbb{C}).\] 
We note that our moduli space $\mathcal{M}_{\mathcal{A}}$ is called an ordered moduli space in \cite{ArtalBartolo2005}. We say a singular point $P$ of $L_1\cup L_2\cup \cdots \cup L_n$ is a \notion{multiple point} of $\mathcal{A}$ if the multiplicity of $P$ is at least $3$.

The following definition is a combination of Nazir and Yoshinaga's original definitions of $C_1$, $C_2$ and simple $C_3$ arrangements.
\begin{Def}
A line arrangement is call $C_{\leq 3}$ if all the multiple points are on at most three lines, say $L_1$, $L_2$ and $L_3$. A  line arrangement is called simple $C_{\leq 3}$ if it is $C_{\leq 3}$ and one of the following condition holds: 
\begin{enumerate}
\item  $L_1\cap L_2\cap L_3\neq \emptyset$ or, 
\item one of $L_1$, $L_2$ and $L_3$ contains at most one more multiple point apart of the possible multiple points $L_1\cap L_2$, $L_2\cap L_3$, and $L_1\cap L_3$.  
\end{enumerate} 
\end{Def}

Here are some examples of arrangements which are not simple $C_{\leq 3}$.

\begin{eg}[MacLane arrangements]\label{MacLane}
A {\em MacLane arrangement} (see Figure \ref{fig:MacLane}) consists of eight lines and eight triple points such that each line passes through  exactly  three triple points. It is not hard to check that the moduli space of MacLane arrangements consists of two points. Representatives of the two points can be defined by the following equations:
\[\scriptstyle xy(x-z)(y-z)(x-y)\Big(x-\dfrac{1\pm\sqrt{-3}}{2}z\Big)\Big(y-\dfrac{1\pm\sqrt{-3}}{2}z\Big)\Big(\dfrac{-1\pm\sqrt{-3}}{2}x-y+z\Big)=0.\]

Since each line passes through three triple points, there are at most seven triple points on three lines. Thus MacLane arrangements can not be simple $C_{\leq 3}$. 
\begin{figure}[htbp]
\centering
\begin{tikzpicture}[domain=4:-2, scale=0.6]
\draw plot (\x, 0) node[left]{$L_1$};
\draw plot (\x, 1) node[left]{$L_2$};
\draw plot (\x, 2) node[left]{$L_3$};

\draw[domain=-2:4] plot (0, \x) node[above]{$L_4$};
\draw[domain=-2:4] plot (1, \x) node[above]{$L_5$};
\draw[domain=-2:4] plot (2, \x) node[above]{$L_6$};

\draw [domain=-2:4] plot (\x, \x) node[right]{$L_7$};

\draw[dashed] (-2, -1)--(-0.5, 0.5);
\draw (-0.5, 0.5)--(1, 2);
\draw (1, 2)--(1.4, 1.2);
\draw[dashed] (1.4, 1.2)-- (1.8,0.4);
\draw (1.8,0.4)--(2.2,-0.4);

\draw (-2,-1) node[left]{$L_8$};
\end{tikzpicture}

\caption{\label{fig:MacLane} MacLane arrangement}
\end{figure}
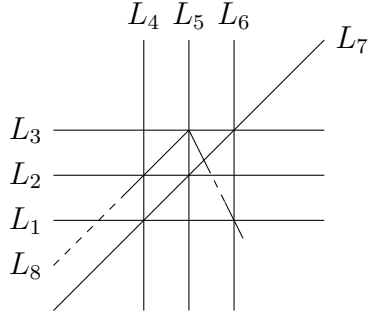
\end{eg}

\begin{eg}[Falk-Sturmfels arrangements (see \cite{Nazir2010} Example 5.2)]\label{Falk-Sturmfels}
Denote by $\gamma_{\pm}=\dfrac{1\pm\sqrt{5}}{2}$ the roots of $x^2-x-1=0$. Let 
\[FS^{\pm}=\{L_i^{\pm}, K_i^{\pm}, H_9^{\pm}, i=1,2,3,4\}\]
be arrangements of nine lines in $\mathbb{P}^2$, where the lines are defined as follows:
 \[\scriptstyle
\begin{array}{llll}
L_1^{\pm} : x=0,  &    
L_2^{\pm} : x=\gamma_\pm(y-z),   & L_3^{\pm} : y=z, & L_4^{\pm} : x+y=z, \\

K_1^{\pm} : x=z, & K_2^{\pm} : x=\gamma_\pm y,  & K_3^{\pm} : y=0, & K_4^{\pm} : x+y=(\gamma_\pm+1)z, \\

 H_9^\pm : z=0. &  &  & 
\end{array}
\]

Arrangements $FS^\pm$ are called {\em Falk-Sturmfels arrangements} (see Figure \ref{fig:Falk-Sturmfels}).  It is known (see for instance, Example 5.2 in \cite{Nazir2010}) that  the moduli space $\mathcal{M}_{L(FS^\pm)}$  consists of 2 points  $\{FS^+, FS^-\}$. 
\begin{figure}[htbp]
\centering
\begin{tikzpicture}[domain=-2: 4, scale=0.7]
\draw[domain=4:-2] plot (\x, 0) node[left]{$K_3^+$};
\draw[domain=4:-2] plot (\x, 1) node[left]{$L_3^+$};
\draw plot (\x, {2/(1+sqrt(5))*\x+1}) node[right]{$L_2^+$};
\draw plot (0, \x) node[above]{$L_1^+$};
\draw plot (1, \x) node[above]{$K_1^+$};
\draw plot ( \x, {2/(1+sqrt(5))*\x}) node[right]{$K_2^+$};
\draw[domain=3:-2] plot (\x, 1-\x) node[left]{$L_4^+$};
\draw[domain=4:(-5+sqrt(5))/2] plot (\x, {(3+sqrt(5))/2-\x})  node[left]{$K_4^+$};
\draw (1, -2.7) node[below]{$FS^+$};
\end{tikzpicture}
\hspace{2em}
\begin{tikzpicture}[domain=-3: 3, scale=0.7]
\draw[domain=3: -3]
 plot (\x, 0) node[left]{$K_3^-$};
\draw[domain=3: -3] plot (\x, 1) node[left]{$L_3^-$};
\draw[domain=-2*(1-sqrt(5)):(1-sqrt(5))] plot (\x, {2/(1-sqrt(5))*\x+1}) node[above]{$L_2^-$};
\draw plot (0, \x) node[above]{$L_1^-$};
\draw plot (1, \x) node[above]{$K_1^-$};
\draw[domain=(3*sqrt(5)-3)/2: (3-3*sqrt(5))/2] plot ( \x, {2/(1-sqrt(5))*\x}) ;
\draw plot ( {(3*sqrt(5) -3)/2}, -3) node[below]{$K_2^-$};
\draw[domain=-2:3] plot (\x, 1-\x) node[right]{$L_4^-$};
\draw[domain=3:(-3-sqrt(5))/2] plot (\x, {(3-sqrt(5))/2-\x})  node[left]{$K_4^-$};
\draw (0, -3.7) node[below]{$FS^-$};
\end{tikzpicture}
\caption{\label{fig:Falk-Sturmfels}Falk-Sturmfels arrangements}
\end{figure}
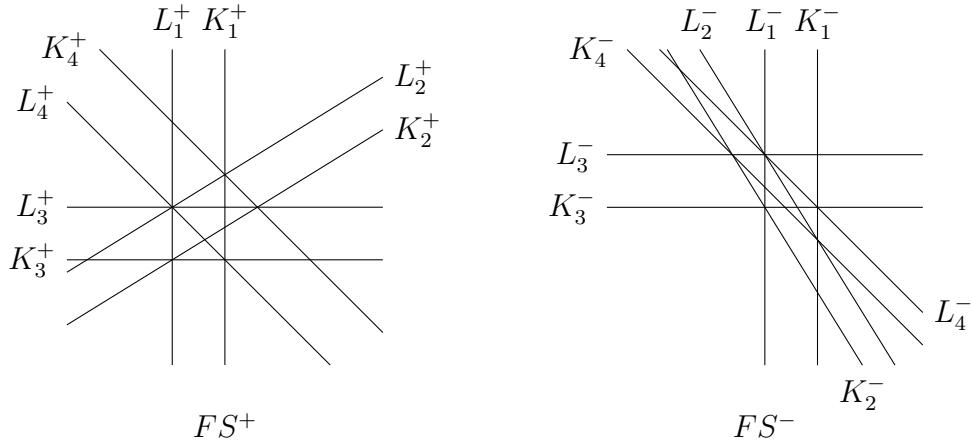
Falk-Sturmfels arrangements are the arrangements of nine lines with one quadruple point, eight triple points, and that one of the lines passes through four triple points.
\end{eg}

\begin{eg}[$\mathcal{A}^{\pm\sqrt{-1}}$]\label{A+-}
Arrangements $\mathcal{A}^{\pm\sqrt{-1}}$ defined by Nazir and Yoshinaga (see \cite{Nazir2010} Example 5.3) consists of nine lines and ten triple points such that there are three lines which do not  intersect at a point and have four triple points on each. Moreover, each of the rest six lines contains exactly three triple points. Those arrangements (see Figure \ref{fig:A+-}) can be defined by following equations:
\[\scriptstyle xy(x-z)(y-z)(x\mp\sqrt{-1}z)(y\mp\sqrt{-1}z)(x-y)((\pm\sqrt{-1}-1)x\pm\sqrt{-1}y+z)((1\mp\sqrt{-1})x+y-z)=0.\]
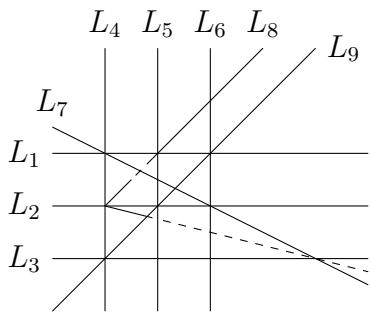
\begin{figure}[htbp]
\centering
\begin{tikzpicture}[domain=5:-1, scale=0.7]
\draw plot (\x, 0) node[left]{$L_3$};
\draw plot (\x, 1) node[left]{$L_2$};
\draw plot (\x, 2) node[left]{$L_1$};
\draw[domain=-1:4] plot (0,\x) node[above]{$L_4$};
\draw[domain=-1:4] plot (1, \x) node[above]{$L_5$};
\draw[domain=-1:4] plot (2, \x) node[above]{$L_6$};
\draw[domain=-1:4] plot (\x, \x) node[right]{$L_9$};
\draw plot (\x, {-1/2*\x+2})  node[above]{$L_7$};
\draw (0.9, 1.9)--(3, 4);
\draw (3,4) node[above]{$L_8$};
\draw[dashed] (0.9, 1.9)--(0.5, 1.5);
\draw (0.5, 1.5)--(0, 1);
\draw[dashed] (1.1, {1-1.1/4})--(5, -1/4);
\draw (0,1)--(0.9, {1-0.9/4});
\end{tikzpicture}
\caption{\label{fig:A+-} $\mathcal{A}^{\pm\sqrt{-1}}$}
\end{figure}
\end{eg}

\begin{lemma}[\cite{Nazir2010} Lemma 4.4]\label{lemma:non-simple}
If a line arrangement is not simple $C_{\leq 3}$, then it has 6 lines, $L_1$, $L_2$, $\dots$, $L_6$, such that $L_1\cap L_2\cap L_3\neq\emptyset$, $L_4\cap L_5\cap L_6\neq\emptyset$ and $(L_1\cup L_2\cup L_3)\cap (L_4\cup L_5\cup L_6)$ consists of 9 distinct double points. 
\end{lemma}
Let $\A_s=\{L_1, L_2,\dots, L_6\}$ be the arrangement which has two triple points  $L_1\cap L_2\cap L_3$ and $L_4\cap L_5\cap L_6$, and nine double points $Q_{ij}=L_i\cap L_{j+3}$, where $i, j\in\{1, 2, 3\}$.

Using Lemma \ref{lemma:non-simple}, one can easily prove that an arrangement of 7 lines is simple $C_{\leq 3}$. It is also not hard to prove the following result.

\begin{prop}[\cite{Nazir2010} Proposition 4.6 ]\label{prop:classification-8-lines}
An arrangement of eight lines is either a simple $C_{\leq 3}$ line arrangement or a MacLane arrangement.
\end{prop}

More generally, Nazir and Yoshinaga proved the following theorem:

\begin{thm}[\cite{Nazir2010} Theorem 3.5]The moduli space $\mathcal{M}_\mathcal{A}$ of simple $C_{\leq 3}$ line arrangements with the fixed intersection lattice $L(\mathcal{A})$  is irreducible. 
\end{thm}

Let $\mathcal{A} = \{L_1, L_2, \dots, L_n\}$ be a line arrangement, 
and $\mathcal{A}'$ $=$ $\{L_1$, $L_2$, $\dots$, $L_{n-1}\}$ be a sub-arrangement.  The following lemma shows when the irreducibility of the moduli space $\mathcal{M}_{\mathcal{A}'}$ will be inherited.

\begin{lemma}[ \cite{Nazir2010}, Lemma 2.4]\label{lemma:extend-irreduciblity}
Assume that the line $L_n$ passes through at most two multiple points of the arrangement $\mathcal{A}$.
Then the moduli space $\mathcal{M}_{\mathcal{A}}$ is a fiber bundle over the moduli space of $\mathcal{M}_{\mathcal{A}'}$. In particular, the moduli space $\mathcal{M}_{\mathcal{A}}$ is irreducible if $\mathcal{M}_{\mathcal{A}'}$ is irreducible.
\end{lemma}

Applying the above lemma to arrangements of 9 lines, we have the following corollary. 

\begin{cor}\label{cor:exists--a-line-with-2-multiple-points}
Let $\mathcal{A}$ be an arrangement of 9 lines. If there is a line in $\mathcal{A}$ which passes through at most two multiple points of $\mathcal{A}$, then either $\mathcal{A}$ contains a MacLane arrangement as a sub-arrangement, or the moduli space $\mathcal{M}_{\mathcal{A}}$ is irreducible.
\end{cor}
\begin{proof}The conclusion follows directly from Proposition \ref{prop:classification-8-lines} and Lemma \ref{lemma:extend-irreduciblity}.
\end{proof}

\section{Classification of arrangements of 9 lines}
For a line arrangement $\mathcal{A}$, we denote by $m_{\mathcal{A}}$ the highest multiplicity of a multiple point of $\mathcal{A}$. We will divide the classification of arrangements of 9 lines into three cases according to the value of $m_{\mathcal{A}}$.

Let $n_r$ be the number of  multiple points of multiplicity $r$.
We first recall two well-known results on the number of multiple points.

\begin{thm}[\cite{Hirzebruch1986}]\label{thm:Hirzebruch}Let $\mathcal{A}$ be an arrangement of $t$ lines in $\mathbb{CP}^2$. Assume that $n_{t}=n_{t-1}=n_{t-2}=0$. Then
\[n_2+\dfrac{3}{4}n_3\geq t +\sum\limits_{r\geq 5}(2r-9)n_r.\]
\end{thm}

\begin{lemma}[see for instance \cite{Hirzebruch1986} ]\label{lemma:intersection-formula}
Let $\mathcal{A}$ be a line arrangement of $n$ lines in $\mathbb{CP}^2$. We have the following intersection formula:
\[\dfrac{n(n-1)}{2}=\sum_{r\geq 2}\Big( n_r\cdot\dfrac{r(r-1)}{2}\Big).\]
\end{lemma}

\subsection{$m_{\mathcal{A}}\geq 5$}
We first observe the following result. 

\begin{prop}\label{prop:r>=5}
Let $\mathcal{A}$ be an arrangements of 9 lines. If $\mathcal{A}$ has a multiple points of multiplicity  at least 5, then the moduli space $\mathcal{M}_\mathcal{A}$ is irreducible. 
\end{prop}
\begin{proof}
Assume that $L_1\cap L_2\cap\cdots\cap L_5\neq \emptyset$. 
There are at most $6$ double points in $L_6\cup L_7\cup L_8\cup L_9$. Then there are at most $7$ multiple points in $L_1\cup L_2\cup\cdots\cup L_5$.  So at least one of the five lines $L_1$, $L_2$, $\dots$, $L_5$ contains only two multiple points. By Corollary \ref{cor:exists--a-line-with-2-multiple-points}, the moduli space $\MA$ is irreducible.
\end{proof}

\subsection{$m_{\mathcal{A}}=4$}
Let $\mathcal{A}$ be an arrangement of 9 lines. In this subsection, we assume that multiple points of $\mathcal{A}$ are at most quadruple points.

\begin{prop}\label{prop:r=4}Assume that each line of $\mathcal{A}$ passes through at least three multiple points and $n_4\geq 1$. Then,  either $\mathcal{M}_{\mathcal{A}}$ is irreducible or  $\mathcal{A}$ is lattice isomorphic to a Falk-Sturmfels arrangement.
\end{prop}
\begin{proof}We will first show that $n_4=1$.

 Let $L_1\cap L_2\cap L_3\cap L_4$ be a quadruple point of $\mathcal{A}$. Since each line passes through at least three multiple points. Then each of $L_1$, $L_2$, $L_3$ and $L_4$ should passes through two more multiple points besides the quadruple point $L_1\cap L_2\cap L_3\cap L_4$. Then, there will be at least 9 multiple points on those four lines. Since multiple points of $\mathcal{A}$ are at most quadruple points and there are $n_4$ quadruple points. Therefore, there should be at least $9-n_4$ triple points on those four lines such that each line passes through at least 3 multiple points. By Theorem \ref{thm:Hirzebruch} and Lemma \ref{lemma:intersection-formula}, we have 
\[36= 6n_4+3n_3+n_2\geq 6n_4+\dfrac{9}{4}n_3+9\geq 6n_4+\dfrac{9}{4}(9-n_4)+9.\]
Solving the inequality, we obtain that $n_4\leq \dfrac{9}{5}<2$. Therefore, by the assumption, we have $n_4=1$.

Now we claim that all triple points should be on the lines passing through the quadruple point.

Let $L_1\cap L_2\cap L_3\cap L_4$ be the quadruple. Suppose, contrary to our claim, that $L_5\cap L_6\cap L_7$ is a triple point which is not on $L_1\cup L_2\cup L_3\cup L_4$. Note that there are at most 7 double points on $L_5\cup L_6\cup L_7\cup L_8\cup L_9$. Then the intersection set $(L_1\cup L_2\cup L_3\cup L_4)\cap (L_5\cup L_6\cup L_7\cup L_8\cup L_9)$ will contain at most 7 triple points which are on $L_1\cup L_2\cup L_3\cup L_4$. However, there should be at least 8 triple points so that each of the four lines $L_1$, $L_2$, $L_3$ and $L_4$ passes through at least three multiple points. Therefore, by the assumption, all triple points must be on the lines passing through the quadruple point.

If $\mathcal{A}$ is simple $C_{\leq 3}$, then the moduli space $\mathcal{M}_\mathcal{A}$ is irreducible. We only need to consider the case that $\mathcal{A}$ is not simple $C_{\leq 3}$. By Lemma \ref{lemma:non-simple}, we know that the arrangement $\mathcal{A}$ has a sub-arrangement $\mathcal{A}_s$. It is not hard to see that the quadruple point should be one of $Q_{ij}$, where $i,j\in \{1,2,3\}$. 

Up to a lattice isomorphism, we may assume that the only quadruple point is $L_1\cap L_4\cap L_7\cap L_8=Q_{11}$.

Since all triple points should be on $L_1\cup L_4\cup L_7\cup L_8$,  then all possible triple points on $L_7$ and $L_8$ should be in the following set of points 
\[\{Q_{22}, Q_{23}, Q_{32}, Q_{33}, L_7\cap L_9, L_8\cap L_9\}.\] 
The following figure (Figure \ref{fig:Fake}) is an example but an excluding one, for $L_6$ passes through only one triple point.
\begin{figure}[htbp]
 \centering
 \begin{tikzpicture}[scale=0.7]
\draw[domain=-1.5:3.5] plot (0, \x) node[above]{$L_4$};
\draw[domain=-1.5:3.5] plot (2, \x) node[above]{$L_5$};
\draw[domain=-1.5:3.5] plot (3, \x) node[above]{$L_6$};

\draw[domain=5:-1] plot (\x, 0) node[left]{$L_3$};
\draw[domain=5:-1] plot (\x, 1) node[left]{$L_2$};
\draw[domain=5:-1] plot (\x, 2) node[left]{$L_1$};

\draw[domain=4:-1] plot (\x, {-\x + 2}) node[above]{$L_7$};
\draw[domain=5:-1] plot (\x, {-0.5*\x+2}) node[left]{$L_8$};
\draw[domain=3.5:-1] plot (\x, \x) node[left]{$L_9$};
\end{tikzpicture}
 \caption{\label{fig:Fake} An excluding arrangement}
 \end{figure}
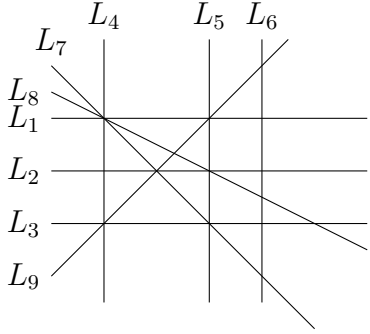 
 
Therefore, each of the lines $L_7$ and $L_8$ will have at least one $Q_{ij}$, where $i, j\in \{2, 3\}$.

 \begin{enumerate}
 \item Assume that each of the lines $L_7$ and $L_8$ passes through exactly one of the points $\{Q_{22},  Q_{23},  Q_{32},  Q_{33}\}$. 
 
 If those two $Q_{ij}$ are on same line, then one of the four lines $L_2$, $L_3$, $L_5$ and $L_6$ will have at most two multiple points. For example, in Figure \ref{fig:Fake}, the line $L_6$ passes through only one multiple points, $L_4\cap L_5\cap L_6$. 
 
 Assume that they are not on the same line. Up to switching labels between  $L_2$ and $L_3$, correspondingly, $L_5$ and $L_6$, we may assume that $Q_{32}\in L_7$ and $Q_{23}\in L_8$. Then either $\{Q_{31}, Q_{13}\}\subset L_9$ or $\{Q_{21}, Q_{12}\}\subset L_9$. Correspondingly, $\{L_2\cap L_7, L_5\cap L_8\}\subset L_9$ or $\{L_3\cap L_7, L_6\cap L_8\}\subset L_9$.   By switching the labels between $L_2$ and $L_3$,  $L_5$ and $L_6$, and $L_7$ and $L_8$, we see that those two arrangements are lattice isomorphic. Moreover, one can check that both arrangements (see Figure \ref{fig:FS-in-classification-1}) are lattice isomorphic to Falk-Sturmfels arrangements. 
 \begin{figure}[htbp]
 \centering
\begin{tikzpicture}[domain=-2:4, scale=0.7]
\draw[domain=-2:4] plot (0, \x) node[above]{$L_4$};
\draw[domain=-2:4] plot (1, \x) node[above]{$L_5$};
\draw[domain=-2:4] plot ({(1+sqrt(5))/2}, \x) node[above]{$L_6$};

\draw[domain=4:-1] plot (\x, 0) node[left]{$L_3$};
\draw[domain=4:-1] plot (\x, 1) node[left]{$L_2$};
\draw[domain=4:-1] plot (\x, {(3+sqrt(5))/2}) node[left]{$L_1$};

\draw[domain=3.54:-1] plot (\x, {-\x + (3+sqrt(5))/2}) node[left]{$L_7$};
\draw[domain={4-sqrt(5)}:{-(5-2*sqrt(5))}] plot (\x, {-(3+sqrt(5))/2*\x+(3+sqrt(5))/2}) node[left=-0.15cm]{$L_8$};
\draw[domain=2.2:-0.8] plot (\x, {(sqrt(5)+1)/2*\x}) node[left]{$L_9$};
\end{tikzpicture}
 \caption{\label{fig:FS-in-classification-1}Falk-Sturmfels arrangement 1}
 \end{figure}
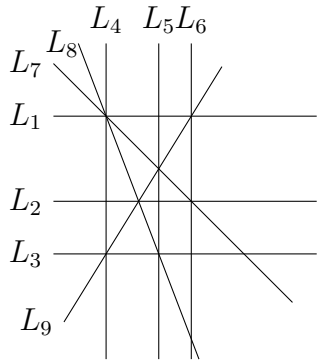
 
\item Assume that one of the lines $L_7$ and $L_8$ passes through two points out of the four points  $Q_{22}$, $Q_{23}$, $Q_{32}$ and $Q_{33}$, but the other one passes only one point out of the four points  $Q_{22}$, $Q_{23}$, $Q_{32}$ and $Q_{33}$. 

Up to a lattice isomorphism,  we may assume that $\{Q_{11}, Q_{22}, Q_{33}\}\subset L_7$ and  $\{Q_{11}, Q_{32}\}\subset L_8$. 
  Then either $L_2\cap L_8\in L_9$ or $L_6\cap L_8\in L_9$.  Otherwise, $L_8$ will have only two multiple points. Correspondingly, $\{Q_{31}, Q_{13}\}\subset L_9$ or $\{Q_{21}, Q_{12}\}\subset L_9$. By first switching the labels between $L_1$ and $L_4$,  $L_2$ and $L_5$, and $L_3\cap L_6$, then switching the labels between $L_2$ and $L_3$, and $L_5$ and $L_6$, we see that those two arrangements are lattice isomorphic. Moreover,  we check that $\mathcal{A}$ (see Figure \ref{fig:FS-in-classification-2}) is also lattice isomorphic to Falk-Sturmfels arrangements.
  \begin{figure}[htbp]
 \centering
 \begin{tikzpicture}[domain=-2:4, scale=0.75]
\draw plot (0, \x) node[above=0.05cm]{$L_4$};
\draw  plot (1, \x) node[above=0.05cm]{$L_5$};
\draw  plot ({(1+sqrt(5))/2}, \x) node[above=0.05cm]{$L_6$};

\draw[domain=4:-2] plot (\x, 0) node[left]{$L_3$};
\draw[domain=4:-2] plot (\x, 1) node[left]{$L_2$};
\draw[domain=4:-2] plot (\x, {(3+sqrt(5))/2}) node[left]{$L_1$};

\draw[domain=-1.2:2.4] plot (\x, {(1+sqrt(5))/2*\x}) node[right]{$L_9$};
\draw[domain=1.75:-0.5] plot (\x, {-(3+sqrt(5))/2*\x+(3+sqrt(5))/2}) node[above=-0.07cm]{$L_8$};
\draw[domain=2.8:-0.8] plot (\x, {-(1+sqrt(5))/2*\x+(3+sqrt(5))/2}) node[left]{$L_7$};
\end{tikzpicture}
 \caption{\label{fig:FS-in-classification-2}Falk-Sturmfels arrangement 2}
 \end{figure}
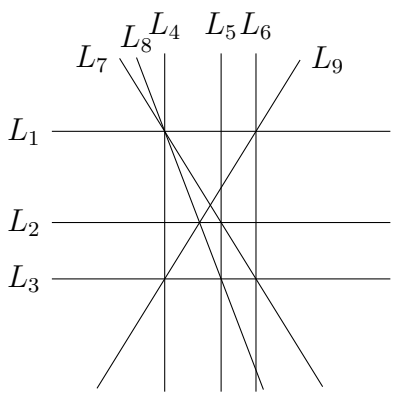
 
  \item Assume that each of $L_7$ and $L_8$ contains two of $\{Q_{22}, Q_{23}, Q_{32}, Q_{33}\}$, then $L_9$ will contain at most two multiple points. 
 \end{enumerate}
 
 Therefore, we conclude that either $\mathcal{M}_\mathcal{A}$ is irreducible or $\mathcal{A}$ is lattice isomorphic to a Falk-Sturmfels arrangement.
 
\end{proof}

\subsection{$m_{\mathcal{A}}=3$}
Now we consider the last case that all multiple points are triple points.   
We will first investigate possible values of $n_3$ such that each line has at least three triple points. Notice that $n_3$ should be no less than 9. On the other hand, we observe the following result.

\begin{lemma}\label{lemma:10-triples}
Let $\mathcal{A}$ be an arrangement of 9 lines. Assume that all multiple points of $\mathcal{A}$ are triple points.  If $\mathcal{A}$ does not contain a MacLane arrangement as a sub-arrangement and is not simple $C_{\leq 3}$. Then, $\mathcal{A}$ has at most 10 triple points.
\end{lemma}

\begin{proof}By  Lemma \ref{lemma:intersection-formula}, to show that $n_3\leq 10$, it is enough to show that $n_2\geq 4$. 

Since $\mathcal{A}$ does not contain a MacLane arrangement,  at most one of the lines $L_7$, $L_8$ and $L_9$ passes through three $Q_{ij}$, where $i, j\in \{1, 2, 3\}$, (defined as above). We may assume that each of the lines $L_7$ and $L_8$ passes through at most two $Q_{ij}$.  By our assumption and Lemma \ref{lemma:non-simple}, the arrangement $\mathcal{A}$ has a sub-arrangement $\mathcal{A}_s$.    

Let $x$ be the number of $Q_{ij}$ which are not in $L_7\cup L_8\cup L_9$. It is clear that $x\geq 2$. Let $y$ and $z$ be the number of double points of $\mathcal{A}$ which are in $L_7\cap (L_1\cup L_2\cup\cdots \cup L_6)$ and $L_8\cap (L_1\cup L_2\cup\cdots \cup L_6)$ respectively. If $y+z\geq 2$, then we have $n_2\geq x+(y+z)\geq 4$. 
 
Assume that $y+z\leq 1$. Then each of  the lines $L_7$ and $L_8$ should passes through exactly  two $Q_{ij}$.  Moreover, $L_7\cap L_8$ must be a triple point in $L_1\cup L_2\cup\cdots \cup L_6$.  We see now the sub-arrangement $\mathcal{A}'=\{L_1, L_2,\dots, L_8\}$ have 7 double points. Without of loss generality, we assume that $L_7\cap L_8$ is on $L_2$. It is not hard to see that the 7 double points of $\mathcal{A}'$  are all on $L_4\cup L_5\cup L_6$. The line $L_9$ can only pass through at most three double points of $\mathcal{A}'$. Therefore, the arrangement $\mathcal{A}$ still has at least 4 double points.
\end{proof}

\begin{remark}It is worth to point out the following fact. 
By Theorem 2.15  \cite{Csima1993}, if our arrangements is  {\em real arrangements}, i.e.  coefficients of the defining equations of the lines are real numbers, then there are at least $60/13>4$ double points. Hence there should be at most 10 triple points. However, there seems no similar result for complex line arrangements.   
\end{remark}

\begin{prop}\label{prop:9-3}
Let $\A$ be an arrangement of $9$ lines with $9$ triple points. Assume that all multiple points of $\mathcal{A}$ are triple points and each line passes through exactly three triple points. Then the moduli space $\MA$ is irreducible.
\end{prop}

\begin{proof}
It is known (see Theorem 2.2.1 \cite{Grunbaum2009}) that $\mathcal{A}$ is lattice isomorphic to one of the three arrangements appearing in Figure \ref{fig:9-3}. 
\begin{figure}[htbp]
\centering
\subfigure[]{\begin{tikzpicture}[scale=0.57]
\draw[domain=-2.5:5]  plot (0, \x) node[above]{$L_4$};
\draw[domain=-2.5:5]  plot (1, \x) node[above]{$L_5$};
\draw[domain=-2.5:5]  plot (3, \x) node[above]{$L_6$};
\draw[domain=4:-3.5] plot (\x, 0) node[left]{$L_1$};
\draw[domain=4:-3.5] plot (\x, 1) node[left]{$L_2$};
\draw[domain=4:-3.5] plot (\x, 2) node[left]{$L_3$};
\draw[domain=-3.5:4] plot (\x, {\x+1}) node[right]{$L_7$};
\draw[domain=-3.5:4] plot (\x, {2*\x/3}) node[right]{$L_8$};
\draw[domain=-3.5:4] plot (\x, {(\x-1)/2}) node[right]{$L_9$};
\end{tikzpicture}}%%%%%%%%%
\hspace{2em}
\subfigure[]{
\begin{tikzpicture}[scale=0.6]
\draw[domain=-2.5:4.5] plot (0, \x) node[above]{$L_4$};
\draw[domain=-2.5:4.5] plot (1, \x) node[above]{$L_5$};
\draw[domain=-2.5:4.5] plot (3, \x) node[above]{$L_6$};
\draw[domain=4:-2] plot (\x, 0) node[left]{$L_1$};
\draw[domain=4:-2] plot (\x, 1) node[left]{$L_2$};
\draw[domain=4:-2] plot (\x, 2) node[left]{$L_3$};
\draw[domain=4:-2] plot (\x, {-2/3*\x+2}) node[left]{$L_8$};
\draw[domain=3.25:-0.25] plot (\x, {2*\x-2}) node[left]{$L_9$};
\draw[domain=3.5:-2] plot (\x, {\x+1}) node[left]{$L_7$};
\end{tikzpicture}
}%%%%%%%%
\hspace{2em}
\subfigure[]{\begin{tikzpicture}[scale=0.75]
\draw[domain=-1:4] plot (0, \x) node[above]{$L_4$};
\draw[domain=-1:4] plot (1, \x) node[above]{$L_5$};
\draw[domain=-1:4] plot (3, \x) node[above]{$L_6$};
\draw[domain=5:-1] plot (\x, 0) node[left]{$L_1$};
\draw[domain=5:-1] plot (\x, 1) node[left]{$L_2$};
\draw[domain=5:-1] plot (\x, 2) node[left]{$L_3$};
\draw[domain=4:-1] plot (\x, \x) node[left]{$L_7$};
\draw[domain=-1:5] plot (\x, {\x/3+1}) node[above]{$L_8$};
\draw[domain=4:-1] plot (\x, {-\x+3}) node[left]{$L_9$};
\end{tikzpicture}}
\caption{\label{fig:9-3} $9_3$ arrangements}
\end{figure}
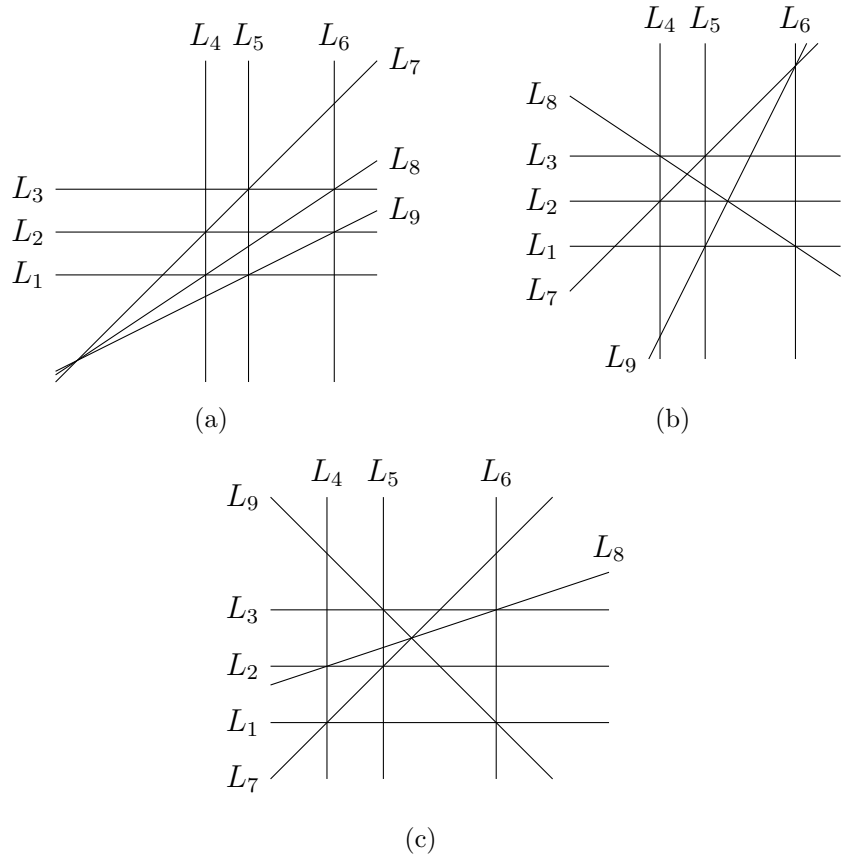
One can check that  the moduli space $\mathcal{M}_{\mathcal{A}}$ is irreducible in each case.
\end{proof}

\begin{prop}\label{prop:r=3}
Let $\mathcal{A}$ be an arrangement of 9 lines with 10 triple points.  Assume that all multiple points of $\mathcal{A}$ are triple points and each line passes through at least three triple points.  If $\mathcal{A}$ is not simple $C_{\leq 3}$, then it is isomorphic to $\mathcal{A}^{\pm\sqrt{-1}}$ (see Figure \ref{fig:A+-}).
\end{prop}
\begin{proof}

Let $a$ be the number of lines that each line passes through 4 triple points and $b$ be the number of lines  that each line passes through 3 triple points. Then $a+b=9$ and $4a+3b=30$. We have that $a=3$ and $b=6$.

 If the three lines with 4 triple points on each of them intersect at  a  triple point, then all the 10 triples should be on them. Consequently, the arrangement is simple $C_{\leq 3}$.
 
Assume that $L_1$, $L_2$ and $L_4$ are the three lines with 4 triple points on each of them and $L_1\cap L_2\cap L_4=\emptyset$.  Then at least two of $\{L_1\cap L_2, L_1\cap L_4, L_2\cap L_4\}$ are triple points. Otherwise there should be at least 11 triple points so that each of lines $L_1$, $L_2$ and $L_4$ will have 4 triple points. So we may assume that  $L_1\cap L_2\cap L_3$ and $L_1\cap L_4\cap L_7$ are triple points. Let $L_4\cap L_5\cap L_6$ be a triple point which is not on $L_1\cup L_2\cup L_3$. 
Then $L_7$ must pass through $L_2\cap L_5$ or $L_2\cap L_6$. Otherwise, $L_2$ will have at most 3 triples. By switching labels of $L_5$ and $L_6$, we may assume that $L_2\cap L_6\cap L_7\neq \emptyset$. Then the two points $Q_{21}$, $Q_{22}$ must be on $L_8\cup L_9$ so that $L_2$ will passes through 4 triple points. We may assume that $Q_{21}\in L_8$ and $Q_{22}\in L_9$. Since the line $L_4$  also passes through 4 triple points, then $Q_{31}$ should be  on $L_9$. Similarly, since the line $L_1$ passes through 4 triple points, then $Q_{13}$ should be on $L_9$ and $Q_{12}$ should be on $L_8$, for that $L_9$ passes through  $Q_{31}$ and $Q_{22}$. Now we have 9 triple points. The last triple point must be $L_3\cap L_7\cap L_8$ so that $L_7$ will passes three triple points. The arrangements with such intersection lattices are noting but $\mathcal{A}^{\pm\sqrt{-1}}$ (see Figure \ref{fig:A+-}).
\end{proof}

\subsection{Classification and applications}
Summarize the above subsections, we have the following theorem:
\begin{thm}
Let $\mathcal{A}$ be an arrangement of nine lines in $\mathbb{CP}^2$. Then $\mathcal{A}$ is in one of the following classes:
\begin{enumerate}
\item[(1)] arrangements whose moduli spaces are irreducible;
\item[(2)] arrangements containing MacLane arrangements (see Example \ref{MacLane});
\item[(3)] Falk-Sturmfels arrangements (see Example \ref{Falk-Sturmfels}).
\item[(4)] arrangements $\mathcal{A}^{\pm\sqrt{-1}}$ (see Example \ref{A+-}).
\end{enumerate}
\end{thm}
\begin{proof}
The classification simply follows from Corollary \ref{cor:exists--a-line-with-2-multiple-points}, and Propositions \ref{prop:r>=5}, \ref{prop:r=4}, \ref{prop:9-3}, and \ref{prop:r=3}.
\end{proof}

As a application, we obtain the following result which generalized a result of \cite{Garber2003} (Theorem 8.3).

\begin{thm}
The fundamental group of the complement of an arrangement of 9 lines is determined by the intersection lattice.
\end{thm}
\begin{proof}
If the moduli space is irreducible, then the fundamental group is determined by the lattice according to the Lattice-isotopy theorem.

It follows from  Example 5.2 \cite{Nazir2010}  (see also section 7.5 \cite{Cohen1997} ) that the fundamental groups $\pi_1(M(FS^+))$ and $\pi_1(M(FS^-))$ are isomorphic. Let $\mathcal{A}_1$ and $\mathcal{A}_2$ be two arrangements containing Maclane arrangements. Then either they are in the same connected component of the moduli spaces, or $\mathcal{A}_1$ and the conjugate of $\mathcal{A}_2$ are in the same connected component. By  Theorem 3.9 \cite{Cohen1997}, the fundamental groups of $\mathcal{A}_1$ and $\mathcal{A}_2$ are isomorphism. According to the same theorem, the fundamental groups of $\mathcal{A}^{+\sqrt{-1}}$ and $\mathcal{A}^{-\sqrt{-1}}$ are isomorphic too.
\end{proof}

\section{Arrangements of 10 lines: an example}
We have seen that there is no Zariski pair of arrangements of 9 lines. But we do not know if there is a Zariski pair of arrangements of 10 lines. To get a Zariski pair,  a naive idea is to add lines to those arrangements whose moduli spaces are disconnected. In general, it is very hard to determine if the resulting pair of arrangements is a Zariski pair.  The following example is a trial.
\begin{eg}
Starting from the Falk-Sturmfels arrangements (see Example \ref{Falk-Sturmfels}), we will construct new arrangements of 10 lines such that the moduli space is disconnected.

We define two line arrangements of 10 lines, called extended Falk-Sturmfels arrangement (see Figure \ref{fig:extended-FS}):
\[\widetilde{FS}^{\pm}=\{L_i^{\pm}, K_i^{\pm}, H_9^{\pm}, H_{10}^{\pm}, i=1,2,3,4\}\]
by adding  lines :
\[H_{10}^\pm : y=(\dfrac{1}{\gamma_\pm}-1)x+ z\]
to $FS^\pm$ respectively.

Notice that $\widetilde{FS}^\pm$  are both fiber-type line arrangements according to Theorem 3.12 \cite{Jiang2001} .
\begin{figure}[htbp]
\begin{center}
\begin{tikzpicture}[domain=4:-2, scale=0.7]
\draw plot (\x, 0) node[left]{$K_3^+$};
\draw plot (\x, 1) node[left]{$L_3^+$};
\draw[domain=-2:4] plot (\x, {2/(1+sqrt(5))*\x+1}) node[right]{$L_2^+$};
\draw[domain=-2:4] plot (0, \x) node[above]{$L_1^+$};
\draw[domain=-2:4] plot (1, \x) node[above]{$K_1^+$};
\draw[domain=-2:4] plot ( \x, {2/(1+sqrt(5))*\x}) node[right]{$K_2^+$};
%\draw plot ( {(3*sqrt(5) +3)/2}, -3) node{$K_2^-$};
\draw[domain=-2:3] plot (\x, 1-\x) node[right]{$L_4^+$};
\draw[domain=(-5+sqrt(5))/2:4] plot (\x, {(3+sqrt(5))/2-\x})  node[right]{$K_4^+$};
\draw[blue] plot (\x, {(2/(1+sqrt(5))-1)*\x+1})  node[left]{$H_{10}^+$};
\draw (1, -2.7) node[below]{$\tilde{FS}^+$};
\end{tikzpicture}
\hspace{2em}
 \begin{tikzpicture}[domain=-3: 3, scale=0.7]
\draw[domain=3: -3] plot (\x, 0) node[left]{$K_3^-$};
\draw[domain=3: -3] plot (\x, 1) node[left]{$L_3^-$};
\draw[domain=(1-sqrt(5)): -2*(1-sqrt(5))] plot (\x, {2/(1-sqrt(5))*\x+1}) node[below right]{$L_2^-$};
\draw plot (0, \x) node[above right]{$L_1^-$};
\draw plot (1, \x) node[above right]{$K_1^-$};
\draw[domain=(3-3*sqrt(5))/2: (3*sqrt(5)-3)/2] plot ( \x, {2/(1-sqrt(5))*\x}) node[below]{$K_2^-$};
\draw[domain=-2:3] plot (\x, 1-\x) node[right]{$L_4^-$};
\draw[domain=(-3-sqrt(5))/2 :3] plot (\x, {(3-sqrt(5))/2-\x})  node[right]{$K_4^-$};
\draw[blue, domain= -4/(2/(1-sqrt(5))-1)):2/(2/(1-sqrt(5))-1)] plot (\x, {(2/(1-sqrt(5))-1)*\x+1}) node[above]{$H_{10}^-$};
\draw (0, -3.8) node[below]{$\tilde{FS}^-$};
\end{tikzpicture}
   \caption{\label{fig:extended-FS}Extended Falk-Sturmfels arrangement}
   \end{center}
\end{figure}
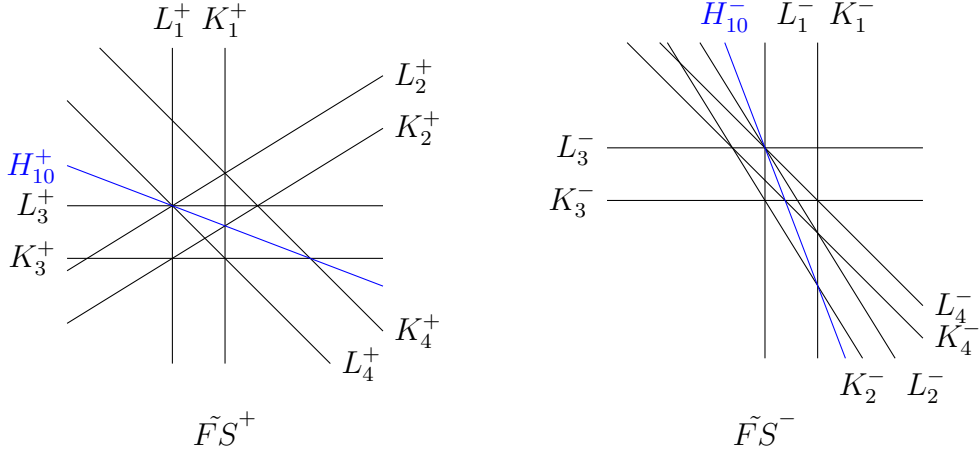

It is not hard to see that $\mathcal{M}_{\widetilde{FS}^\pm}\cong\mathcal{M}_{{FS}^\pm}$. In fact, the line $H_{10}^+$ (respectively, $H_{10}^-$) is always passing through three points of $L(FS^\pm)$: 
$L^+_1\cap L^+_2$, $K^+_1\cap K^+_2$ and $K_3^+\cap K_4^+$ (respectively, $K_2^-\cap K_4^-$, $K_3^-\cap K_4^-$ and $K_1^-\cap K_2^-$).

This pair of arrangements has been studied by Artal, Carmona, Cogolludo and Marco. They show ( \cite{ArtalBartolo2005} Theorem 4.19) that there is no order-preserving homeomorphism between the pairs $(\mathbb{P}^2, \widetilde{FS}^+)$ and  $(\mathbb{P}^2, \widetilde{FS}^-)$. 
 Here we present an explicit diffeomorphism between the complements $M(\widetilde{FS}^+)$ and $M(\widetilde{FS}^-)$. In fact, by Example 5.2 \cite{Nazir2010}, we know that there is an automorphism $A\in PGL(\mathbb{C}^3)$ of $\mathbb{CP}^2$, 
\[A:= \begin{pmatrix}
-\gamma_- & -1 & 0\\
-\gamma_- & 0 & 0\\
\gamma_- & 1 & 1
\end{pmatrix},\]
acting on the right to points $[x , y , z]$ in the projective space $\mathbb{P}^2$ (as matrices product) which sends
\[
\begin{array}{llll}
L_1^+\mapsto L_3^- ,  &   L_2^+\mapsto L_4^- ,   &  L_3^+\mapsto L_2^- , & L_4^+\mapsto L_1^-  \\
K_1^+\mapsto K_3^- ,  &   K_2^+\mapsto K_4^- ,   &  K_3^+\mapsto K_2^- , & K_4^+\mapsto K_1^-  \\
 H_9^+\mapsto H_9^- &  &  & 
\end{array}.
\]

To see that $A$ induces a diffeomorphism between $M(\widetilde{FS}^+)$ and $M(\widetilde{FS}^-)$, it suffices to show that the automorphism $A$ sends $H_{10}^+$ to $H_{10}^-$. 

Recall that $\gamma_{\pm}=\dfrac{1\pm\sqrt{5}}{2}$.  One can check that for any point $P:=[x, (\dfrac{1}{\gamma_+}-1)x+z, z]$ on $H_{10}^+$, the image $P\cdot A$ is a point on $H_{10}^-$. In fact, 
\[
\begin{pmatrix} x & (\dfrac{1}{\gamma_+}-1)x+z & z \end{pmatrix}\cdot A\cdot \begin{pmatrix}\dfrac{1}{\gamma_-}-1 \\ -1 \\ 1 \end{pmatrix}\equiv 0.\]
Therefore, the pair $(\widetilde{FS}^+, \widetilde{FS}^-)$ is not a Zariski pair. 
\end{eg}

From this example, we see that moduli spaces of fiber-type projective line arrangements do not have to be connected.  In fact, we can produce infinitely many fiber-type projective line arrangements whose moduli spaces are disconnected. 
On the other hand, we do not know if fundamental groups of complements of fiber-type projective line arrangements are determined by intersection lattices. 

\bibliographystyle{alpha}

%\bibliography{bibliography} 

\end{document}